\documentclass[10pt,a4paper]{amsart}
\usepackage{amssymb}
 \usepackage[latin1]{inputenc}
 \usepackage[english]{babel}
 \usepackage{amsfonts}
 \usepackage{amsmath}
 \usepackage{amsthm}
 \usepackage[dvips]{graphicx}
 \DeclareGraphicsExtensions{.pdf,.png,.jpg,.eps}
 \usepackage{epstopdf}
 \usepackage{fancyhdr}
 \usepackage{colortbl}
 \usepackage{color}
 \usepackage[pdftex]{hyperref}
\numberwithin{equation}{section}
\newtheorem{theorem}{Theorem}[section]

\theoremstyle{remark}
\newtheorem{remark}{Remark}[section]

\theoremstyle{definition}

\newtheorem{asumption}{Asumption}[section]
\newtheorem*{acknowledgments}{Acknowledgments}

\newcommand{\triple}[1]{{|\!|\!|#1|\!|\!|}}
\newcommand{\R}{\mathbb{R}}

\newcommand{\Z}{\mathbb{Z}}

\begin{document}


\title[$L^p$-$L^q$ estimates for Electromagnetic Helmholtz equation. Singular potentials]
{$L^p$-$L^q$ estimates for Electromagnetic Helmholtz equation. Singular potentials}

\author{Andoni Garcia}
\address{Andoni Garcia: Department of Mathematics and Statistics, P.O. Box 35 (MaD), FI-40014, University of Jyv\"{a}skyl\"{a}, Finland}
\email{andoni.a.garcia@jyu.fi}



\begin{abstract}
In space dimension $n\geq3$, we consider the electromagnetic Schr\"odinger Hamiltonian $H=(\nabla-iA(x))^2+V$ and the corresponding Helmholtz equation
  \begin{equation*}
   (\nabla-iA(x))^2u+u+V(x)u=f\quad \text{in}\quad \mathbb{R}^n,
  \end{equation*}
where the magnetic and electric potentials are allowed to have singularities at the origin and decay at infinity. We extend the well known $L^p$-$L^q$ estimates for the solution of the free Helmholtz equation to the case when the electromagnetic hamiltonian $H$ is considered. This work extends the results that appear in \cite{G}.
 \end{abstract}

\date{\today}


\maketitle

\section{Introduction}\label{sec:intro}
In this paper we extend our previous results appearing in \cite{G}, where $L^p$-$L^q$ estimates for the solution of the electromagnetic Helmholtz equation were proven, to the case of singular potentials $A$ and $V$. We want to preserve the behavior of the potentials at infinity, that is, consider potentials with short-range decay without assuming smallness. Our goal will be to extend the well known $L^{p}$-$L^{q}$ estimates for the free Helmholtz equation given in \cite{KRS}, \cite{CS}, \cite{Gut} and \cite{Gut1}, to the case when we perturb the equation with an electromagnetic potential with singularities at the origin. More precisely, conditions on the electric and the magnetic part of the potential will be given in order to ensure that the estimates remain true. The $L^p$-$L^q$ estimates for the free Helmholtz equation are the following
\begin{equation}\label{eq:lpestimates}
\|u\|_{L^q(\R^n)}\leq C\|f\|_{L^q(\R^n)},
\end{equation}
where $u$ is a solution of
\begin{equation}
\Delta u + (\tau\pm i\varepsilon) u=-f\quad\tau,\epsilon >0.
\end{equation}
The exponents $p$ and $q$ in \eqref{eq:lpestimates} have to verify some specific conditions that will be specified later on. Here $C$ can depend on $\tau$, $p$, $q$ and $n$ and is independent of $\epsilon$.

In the first part of the paper we will prove that the existing results for the free Helmholtz equation can be extended to the perturbed equation by imposing conditions on the admissible singularity at the origin and decay at infinity for the potentials. This can be done without assuming smallness at infinity, neither for the electric part, nor for the magnetic part. It will be concluded that the range for $p$ and $q$ where estimates \eqref{eq:lpestimates} are true is not the same as the one for the free Helmholtz equation. Therefore this motivates the second part of the paper in which, by setting $A\equiv 0$, we consider the Helmholtz equation with electric potential and prove estimates \eqref{eq:lpestimates} in the same region of boundedness of the free equation.

Therefore, we consider the electromagnetic Schr\"odinger hamiltonian $H$ of the form
\begin{equation}
H=(\nabla-i A(x))^2 + V(x),
\end{equation}
and the Helmholtz equation in dimensions $n\geq3$, namely,
\begin{equation}\label{eq:Helmmagneticin}
	(\nabla-iA(x))^2u+u+V(x)u=f\quad \text{in}\quad \mathbb{R}^n.
\end{equation}
Here, $A:(A^1,\dots, A^n):\mathbb{R}^n\rightarrow\mathbb{R}^n$ is the magnetic potential and $V(x):\mathbb{R}^n\rightarrow\mathbb{R}$ is the electric potential. Since now on, we denote by
\begin{equation}\label{eq:maggrad}
	\nabla_{A}=\nabla-iA,\qquad\Delta_{A}=\nabla_{A}^2.
\end{equation}
The magnetic potential $A$ is a mathematical construction which describes the interaction of particles with an external magnetic field. The magnetic field $B$, which is the physically measurable quantity, is given by
\begin{equation}\label{eq:B}
  B\in\mathcal M_{n\times n},
  \qquad
  B=DA-(DA)^t,
\end{equation}
i.e. it is the anti-symmetric gradient of the vector field $A$ (or, in geometrical terms, the differential $dA$ of the 1-form which is standardly associated to $A$). In dimension $n=3$ the action of $B$ on vectors is identified with the vector field $\text{curl} A$, 
\begin{equation}\label{eq:B3}
  Bv=\text{curl} A\times v
  \qquad
  n=3,
\end{equation}
where the cross denotes the vectorial product in $\R^3$.

\noindent We also define the trapping component of $B$ as
\begin{equation}\label{eq:Btau}
B_\tau(x)=\frac{x}{|x|} B(x),\qquad (B_\tau)_j=\sum_{k=1}^n\frac{x_k}{|x|}B_{kj}
\end{equation}
and we say that $B$ is non trapping if $B_\tau$=0. Observe that in dimension $n=3$ it coincides with 
\begin{equation}
B_\tau(x):=\frac{x}{|x|}\times \text{curl} A(x).
\end{equation}
Hence, $B_\tau(x)$ is the projection of $B$ on the tangential space in $x$ to the sphere of radius $|x|$, for $n=3$. Observe also that $B_\tau\cdot x=0$ for any $n\geq2$, therefore $B_\tau$ is a tangential vector field in any dimension and we call it the tangential component of the magnetic field $B$.

In order to ensure the self-adjointness of $H$ we require some local integrability conditions on the potentials. We will assume 
\begin{equation}\label{eq:integrability}
A_j\in L^2_{loc}(\R^n), \quad V\in L^1_{loc}(\R^n),\quad\int V|u|^2 dx\leq \nu\int|\nabla u|^2 dx,\quad 0<\nu<1.
\end{equation}
From this assumptions it can be concluded (See \cite{Z1}) that $H$ is self-adjoint on $L^2(\R^n)$ with form domain
\begin{equation*}
D(H)=\{f\in L^2(\R^n): \int |\nabla_A f|^2 dx - \int V |f|^2 dx < \infty\}.
\end{equation*}
Note that by \eqref{eq:integrability} then $D(H)$ is equivalent to the Hilbert space
\begin{equation*}
H^1_A(\R^n)=\{f\in L^2(\R^n) : \int |\nabla_A f|^2 dx < \infty\}.
\end{equation*}
Since the spectrum of a self-adjoint operator is real, we conclude the existence of solution of 
\begin{equation}
Hu+(1\pm i\epsilon)u = f\quad \text{in}\quad\R^n,\quad\epsilon\neq0,
\end{equation}
for any $f\in L^2(\R^n)$ and $u\in H^1_A(\R^n)$. See \cite{IK}, \cite{LS}, \cite{AHS} or \cite{CFKS} for more details about the self-adjointness of the electromagnetic Schr\"odinger operator $H$.

The main difficulty in order to obtain the estimates is that when the magnetic laplacian is expanded we have to deal with a first order term, namely $A\cdot\nabla$, and it is well known that there are no $L^{p}$-$L^{q}$ estimates for the gradient of the solution of the free Helmholtz equation,
\begin{equation}
\Delta u+u=-f\quad \text{in}\quad \mathbb{R}^n.
\end{equation}
The main result of the paper appears in Section \ref{sec:magnetic}, concretely Theorem \ref{thm:elecmag}. 

We will proceed in the following way. Let us consider the modified Helmholtz equation with electromagnetic potential and fixed frequency $\tau=1$. It is given by
\begin{equation}\label{eq:Helmagmo}
	(\nabla-iA(x))^2u+(1\pm i\epsilon)u+V(x)u=f\quad \text{in}\quad \mathbb{R}^n,\quad \epsilon\neq0.
\end{equation}
\begin{remark}
For convenience we will deal only with the case $\tau=1$, in contrast with the case of general $\tau>0$.
\end{remark}
In order to derive the estimates for the solution of  \eqref{eq:Helmmagneticin}, $L^{p}$-$L^{q}$ estimates, independent of $\epsilon$, will be obtained for the solution of the modified Helmholtz equation with electromagnetic potential \eqref{eq:Helmagmo}. The independence of $\epsilon$ for the estimates will imply that these will remain true for the solution of \eqref{eq:Helmmagneticin}. This is true due some results in \cite{Z}.

Our method is a mixture of a priori estimates and perturbative arguments. This is what allows us to avoid smallness conditions in the potentials. Similar arguments have been used in the setting of the free Schr\"odinger equation, as can be seen in \cite{BPST} and \cite{DFVV}. Along the proof our basic tools will be the corresponding  $L^{p}$-$L^{q}$ estimates and a $L^{2}$-local estimate for the solution of the free Helmholtz equation, together with an a priori estimate for the solution of the modified Helmholtz equation with electromagnetic potential \eqref{eq:Helmagmo}.

Let us introduce some notation. For $f:\mathbb{R}^n\to\mathbb{C}$ we define the Morrey-Campanato norm as 
\begin{equation}\label{eq:triplenorma}
\triple{f}^2:=\sup_{R>0}\frac{1}{R}\int_{|x|\leq R}|f|^2dx.
\end{equation}
Moreover, we denote, for $j\in \mathbb{Z}$, the annulus C(j) by 
\begin{equation*}
C(j)=\{x\in\mathbb{R}^n:2^j\leq|x|\leq2^{j+1}\},
\end{equation*}
\begin{equation}\label{eq:N}
	N(f):=\sum_{j\in \mathbb{Z}}\left(2^{j+1}\int_{C(j)}|f|^2dx\right)^{1/2},
\end{equation}
and we easily see the duality relation
\begin{equation*}
\int fgdx\leq\triple{g}\cdot N(f).
\end{equation*}
These norms were introduced by Kenig, Ponce and Vega in \cite{KPV}.



Concerning the perturbative part of our argument, it is necessary to remind the results that are true for the free Helmholtz equation. Firstly, we are going to state the result concerning the $L^{p}$-$L^{q}$ estimates which appears in \cite{KRS}, \cite{Gut} and \cite{Gut1}. 
Let 
\begin{align*}
&A=\left(\frac{n+3}{2n},\frac{n-1}{2n}\right),\qquad A'=\left(\frac{n+1}{2n},\frac{n-3}{2n}\right)\\
&
B=\left(\frac{n^2+4n-1}{2n(n+1)},\frac{n-1}{2n}\right),\qquad B'=\left(\frac{n+1}{2n},\frac{n^2-2n+1}{2n(n+1)}\right)\nonumber\\
\nonumber
\end{align*}
and $\Delta(n)$, for $n\geq3$, is the set of points of $[0,1]\times[0,1]$ given by

\begin{equation}
	\Delta(n)=\left\{\left(\frac{1}{p},\frac{1}{q}\right)\in[0,1]^2:\frac{2}{n+1}\leq \frac{1}{p}-\frac{1}{q}\leq\frac{2}{n}, \frac{1}{p}>\frac{n+1}{2n}, \frac{1}{q}<\frac{n-1}{2n}\right\}.
\end{equation}
The set $\Delta(n)$ is the trapezium $A$$B$$B'$$A'$ with the closed line segments $AB$ and $B'A'$ removed (See Figure \ref{fig:deltan}). 


\begin{figure}[!htb]
\begin{center}
\includegraphics[]{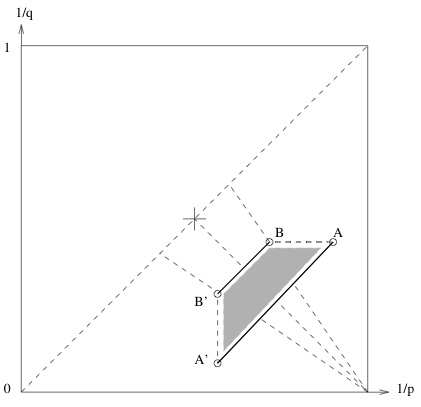}
\caption{$\Delta(n)$, $n\geq3$.}
\label{fig:deltan}
\end{center}
\end{figure}
Now, we are in conditions to recall the existing result for the Helmholtz equation with constant coefficients. 
\begin{remark}
In \cite{Gut} (See also \cite{Gut1}), estimates for the solution of the equation perturbed with generals $\tau>0$ and $\epsilon>0$ are given, namely, the equation reads as follows,
\begin{equation}
\Delta u+(\tau+i\epsilon)u=-F,\quad\tau,\epsilon>0.
\end{equation}
\noindent Recall that we will only deal with the case of fixed frequency $\tau=1$.
\end{remark}
The result is the following one.

\begin{theorem}\label{thm:Helmfree}
	Let u be a solution of 
	\begin{equation}
		\Delta u+(1+i\epsilon)u=-F,\quad\epsilon>0.
	\end{equation}
	Then, there exists a constant C, independent of $\epsilon$, such that 
	\begin{equation}\label{eq:estfree}
		\|u\|_{L^q(\mathbb{R}^n)}=\|(\Delta+(1+i\epsilon))^{-1}F\|_{L^q(\mathbb{R}^n)}\leq C\|F\|_{L^p(\mathbb{R}^n)}
	\end{equation}
	when $(\frac{1}{p},\frac{1}{q})\in\Delta(n)$, $n\geq3$.
\end{theorem}
As we mentioned, another tool that will be crucial in the proof is an $L^2$-local estimate, which bounds the Morrey-Campanato norm of the solution of the free equation, defined in \eqref{eq:triplenorma}, in terms of the $L^p$ norm of the RHS data. This theorem also appears in \cite{RV}, \cite{Gut} and \cite{Gut1}. The statement is the following.

\begin{theorem}\label{thm:triple}Let u be a solution of 
\begin{equation}
		\Delta u+(1+i\epsilon)u=-F,\quad \epsilon>0.
\end{equation}
If 
\begin{itemize}
\item[\textit{(i)}]
$n=3$\quad\text{and} \quad$\frac{1}{4}\leq\frac{1}{p}-\frac{1}{2}<\frac{1}{2}$,\quad or
\item[\textit{(ii)}]
$n\geq4$\quad\text{and} \quad$\frac{1}{n+1}\leq\frac{1}{p}-\frac{1}{2}\leq\frac{3}{2n}$,\quad
\end{itemize}
then, there exists a constant $C$, independent of $\epsilon$, such that
\begin{equation}\label{eq:tripletrun}
	\sup_{R>0}\left(\frac{1}{R}\int_{B_{R}}|u(x)|^2dx\right)^{1/2}\leq C\|F\|_{L^p(\mathbb{R}^n)}.
\end{equation}
\end{theorem}

Once we have established the theorems that we use in our perturbative argument it remains to state the a priori estimate that we are going to consider. We have that, given precise conditions on the admissible singularity at the origin for the electromagnetic potential, and the decay at infinity as well, a Morrey-Campanato type estimate for the solution of the perturbed Helmholtz equation holds for $n\geq3$. This can be seen in \cite{Z}. 

We give now the main assumptions that must verify the potentials $V(x)$ and $A(x)$ in \cite{Z}. The electric potential will be decomposed as $V(x)=V_1(x)+V_2(x)$, where $V_1$ is a long range potential and $V_2$ is a short range one which might be singular.

Let $V_1(x)$, $A_j(x)$, $j=1,\dots,n$, $V_2(x)$ be real-valued functions, $r_0\geq1$ and $\mu>0$. For $n\geq3$, if $|x|>r_0$ we assume
\begin{equation}\label{eq:potentialinf}
\frac{|V_1(x)|}{|x|}+(\partial_r V_1(x))_{-} + |B_\tau(x)|+|V_2(x)|\leq\frac{c}{|x|^{1+\mu}},
\end{equation}
for some $c>0$, where $\partial_rV_1=\frac{x}{|x|}\cdot \nabla V_1$ is considered in the distributional sense and $(\partial_r V_1)_{-}$ denotes the negative part of $\partial_r V_1$. On the other hand, we require
\begin{equation}\label{eq:V1}
V_1(x)=(\partial_r V_1(x))_{-}=0\quad\text{if}\quad|x|\leq r_0,
\end{equation}
and
\begin{equation}\label{eq:V2}
|V_2(x)|\leq\frac{c}{|x|^{2-\alpha}}\quad\text{if}\quad|x|\leq r_0,\quad\alpha>0,
\end{equation}
for some $c>0$.

\noindent If $n>3$, we consider
\begin{equation}\label{eq:Bgret3}
|B|\leq\frac{C^*}{|x|^2}\quad|x|\leq r_0,
\end{equation}
for some $C^*>0$ small enough. Finally, in dimension $n=3$ we assume
\begin{equation}\label{eq:B3}
|B|\leq\frac{c}{|x|^{2-\alpha}}\quad|x|\leq r_0,\quad\alpha>0,
\end{equation}
for some $c>0$.

\begin{remark}
For simplicity, from now on we take $r_0=1$.
\end{remark}

\noindent Finally, let us define the following function 
\begin{equation}\label{eq:chi}
\chi(n)=
\begin{cases}
1\qquad\text{if}\quad n=3,\\
0\qquad\text{if}\quad n\neq3.
\end{cases}
\end{equation}
The theorem reads as follows (see \cite{Z}).

\begin{theorem}\label{thm:apriori}
Let $n\geq3$, $f$ such that $N(f)<\infty$ and $u\in H^1_A(\R^n)$ a solution of
\begin{equation}
	(\nabla + i A(x))^{2}u + V_1(x)u + V_2(x)u+(1\pm i\epsilon)u=f,\quad\epsilon\neq0,
\end{equation}
where $V_1$, $V_2$ and $A(x)$ satisfy assumptions \eqref{eq:potentialinf}-\eqref{eq:V2}, \eqref{eq:Bgret3} for $n>3$, and \eqref{eq:potentialinf}-\eqref{eq:V2}, \eqref{eq:B3} for $n=3$. Then, there exists a constant $C$, which depends uniformly in $\epsilon$, such that the following a priori estimate holds
\begin{align}\label{eq:estimpriori}
\triple{u}^{2}+\triple{\nabla_{A}u}^{2}+\int_{|x|\leq1}\frac{|\nabla_{A}u|^{2}}{|x|}dx&+\sup_{R>0}\frac{1}{R^{2}}\int_{|x|=R}|u|^{2}d\sigma_{R}(x)
\\
&
+(1-\chi(n))\int_{\R^{d}}\frac{|u|^{2}}{|x|^{3}}dx\leq CN(f)^{2}.
\nonumber
\end{align}	
\end{theorem}

\begin{remark}
In order to prove the a priori estimate \eqref{eq:estimpriori}, in \cite{Z} is sufficient to impose conditions on the trapping component of the magnetic field $B_\tau$, not on the whole $B$. More precisely, they depend on the dimension. For $n>3$, $B_\tau$ must satisfy with $C^*<\sqrt{(n-1)(n-3)}$,
\begin{equation}
|B_\tau|\leq\frac{C^*}{|x|^2}\quad|x|\leq 1,
\end{equation}
and if $n=3$
\begin{equation}
|B_\tau|\leq \frac{c}{|x|^{2-\alpha}}\quad |x|\leq 1\quad c,\alpha >0.
\end{equation}
\end{remark}

\begin{remark}
The a priori estimates from Theorem 2.1 and Theorem 2.2 in \cite{Z} are a bit different from the one that we present in \eqref{eq:estimpriori}. In the left hand side it appears the term
\begin{equation*}
\int \frac{|\nabla^{\bot}_A u|^2}{|x|}dx,
\end{equation*}
where $\nabla^{\bot}_A$ denotes the tangential component of the magnetic gradient $\nabla_A$ defined in \eqref{eq:maggrad}. In principle we can not conclude anything from this for the term
\begin{equation*}
\int_{|x|\leq1}\frac{|\nabla_{A}u|^{2}}{|x|}dx.
\end{equation*}
The estimate that we need for this term appears in \cite{Z1}, Lemma 2.4.1.
\end{remark}

%

\begin{remark}In \cite{Z}, 
Theorem 1.3 is proved for all $\tau\geq\tau_{0}>0$ and $\epsilon\in(0,\epsilon_{1})$. We only deal with the particular case $\tau=1$. The case $\tau=0$ requires more decay on the potentials, as can be seen in \cite{F}.
\end{remark}

Once we have described all the tools which are going to be used, it is necessary to introduce the region where we are able to extend the known results for the free Helmholtz equation to the case when electromagnetic perturbations are considered. During the discussion, it will appear a subregion of $\Delta(n)$, for $n\geq3$, which will be denoted by $\Delta_{0}(n)$, given by
\begin{equation}
	\Delta_{0}(n)=\left\{\left(\frac{1}{p},\frac{1}{q}\right)\in\Delta(n):\frac{1}{n+1}\leq \frac{1}{p}-\frac{1}{2},\frac{1}{n+1}\leq \frac{1}{2}-\frac{1}{q}\right\}.
\end{equation}
The set $\Delta_{0}(n)$ is the solid triangle determined by the points $Q$, $Q'$ and $Q''$ (See Figure \ref{fig:deltan0}).

This will be the region of boundedness for the perturbed Helmholtz equation.

\begin{figure}[!htb]
\begin{center}
\includegraphics[]{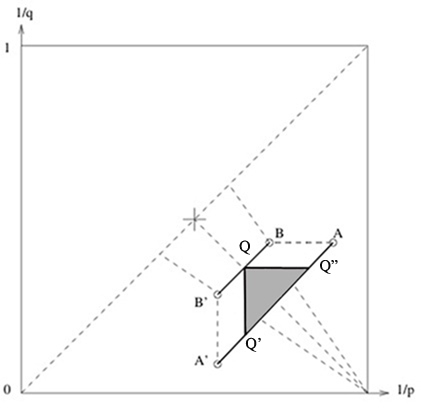}
\caption{$\Delta_{0}(n)$, $n\geq3$.}
\label{fig:deltan0}
\end{center}
\end{figure}

\begin{remark}
For the case of the perturbed electromagnetic equation, we are not able to obtain a positive result of boundedness for the whole region $\Delta(n)$, since we have not control for the gradient term, namely $A\cdot\nabla$, outside $\Delta_{0}(n)$. However, if we consider $A\equiv 0$, we can probe the estimates in the whole $\Delta(n)$ by imposing more decay on $V$ outside $\Delta_0(n)$.
\end{remark}


\section{Electromagnetic Helmholtz Equation}\label{sec:magnetic}


In this section we give the precise statement and the proof of the theorem, where we extend the $L^{p}$-$L^{q}$ estimates for the solution of the electromagnetic Helmholtz equation. The theorems and the notations which will be used along the proof were presented in Section \ref{sec:intro}. We give first the result for the Helmholtz equation with electromagnetic potential and afterwards, by setting $A\equiv 0$ we extend the result to the electric case. 

Let us start by considering the Helmholtz equation with electromagnetic potential in dimensions $n\geq3$. It reads,

\begin{equation}\label{eq:Helmmagnetic}
	(\nabla-iA(x))^2u+u+V(x)u=f\quad \text{in}\quad \mathbb{R}^n,
\end{equation}
where $A:(A^1,\dots, A^n):\mathbb{R}^n\rightarrow\mathbb{R}^n$ is the magnetic potential and $V(x):\mathbb{R}^n\rightarrow\mathbb{R}$ is the electric potential.

We will also assume that the magnetic potential $A$ is divergence free, or in other words, $A$ satisfies the so called Coulomb gauge condition
\begin{equation}
 \nabla\cdot A=0.
\end{equation}
We will prove $L^{p}$-$L^{q}$ estimates for the solution of the equation \eqref{eq:Helmmagnetic}.
In order to do that we will consider the solution of \eqref{eq:Helmmagnetic} as the solution of the modified Helmholtz electromagnetic equation, 
\begin{equation}\label{eq:helma}
	(\nabla-iA(x))^2u+(1\pm i\epsilon)u+V(x)u=f\quad \text{in}\quad \mathbb{R}^n,\quad\epsilon\neq0,
\end{equation}
via the limiting absorption principle, by taking the limit of the solution of \eqref{eq:helma} when $\epsilon$ goes to 0. We will obtain the corresponding $L^{p}$-$L^{q}$ estimates, independent of $\epsilon$, for the solution of \eqref{eq:helma}, so these will remain true for the solution of \eqref{eq:Helmmagnetic}. This procedure is justified by the results appearing in \cite{Z}.

The goal is to determine the region of $p$ and $q$ where the solution of \eqref{eq:helma} satisfies $L^{p}$-$L^{q}$ estimates, namely, 
\begin{equation}
	\|u\|_{L^q(\mathbb{R}^n)}\leq C \|f\|_{L^p(\mathbb{R}^n)}.
\end{equation}
with $C$ independent of $\epsilon$.


We will impose the following conditions on the potentials.


\begin{asumption} Let us assume that the electric potential $V$ is such that
\begin{equation}
V(x)=V_2(x),
\end{equation}
 that is, $V$ satisfies the assumptions of Theorem \ref{thm:apriori},
\begin{equation}
	V(x)=V_1(x)+V_2(x),
\end{equation}
with $V_1\equiv0$.

In order to use \eqref{eq:estimpriori}, the magnetic potential $A$ and the electric potential $V$ must satisfy the assumptions \eqref{eq:potentialinf}-\eqref{eq:V2} and \eqref{eq:Bgret3} for $n\geq4$ and \eqref{eq:potentialinf}-\eqref{eq:V2} and \eqref{eq:B3} when $n=3$.

Moreover, $A$ and $V$ must verify the following conditions on the admissible singularity at the origin and the decay at infinity. They read as
\begin{equation}\label{eq:H1}
{\bf{(H1)}}
\begin{cases}
|A(x)|\leq\frac{C}{|x|^{\beta}},\quad\beta<\frac{n+3}{2(n+1)},\quad|x|\leq1,\\
|V(x)|\leq\frac{C}{|x|^{\gamma}},\quad\gamma<\frac{3n+5}{2(n+1)},\quad|x|\leq1,\\
|A(x)|\leq\frac{C}{|x|^{1+\mu}},\quad\mu>0,\quad|x|>1,\\
|V(x)|\leq\frac{C}{|x|^{1+\mu}},\quad\mu>0,\quad|x|>1.
\end{cases}
\end{equation}


\begin{remark}\label{rm:smallness}
The constant $C$ that appears in ${\bf(H1)}$ and which measures the smallness near the origin of the potentials, does not need to be small as in Theorem \ref{thm:apriori}, since that we are far away from the critical admissible singularity for the magnetic potential $A$ and the electric potential $V$ given by the conditions \eqref{eq:V2} and \eqref{eq:Bgret3} for $n\geq4$ and \eqref{eq:V2} and \eqref{eq:B3} if $n=3$.
\end{remark}

\end{asumption}

The main result of the paper is the following. This is the natural generalization of the Theorem 2.1 in \cite{G} to the case of singular potentials.

\begin{theorem}
\label{thm:elecmag}
Let $u \in H^1_A(\R^n)$ be a solution of 
\begin{equation}\label{eq:Helmmag}
	(\nabla-iA(x))^2u+(1\pm i\epsilon)u+V(x)u=f\quad \text{in}\quad \mathbb{R}^n,\quad n\geq3,\quad\epsilon\neq0.
\end{equation}
where $V$ and $A$ satisfy assumptions \eqref{eq:potentialinf}-\eqref{eq:V2}, \eqref{eq:Bgret3} and $({\bf H1})$ for $n=3$, and \eqref{eq:potentialinf}-\eqref{eq:V2}, \eqref{eq:B3} and $({\bf H1})$  for $n\geq4$. Then, there exists a constant C, independent of $\epsilon$, such that 
	\begin{equation}
		\|u\|_{L^q(\mathbb{R}^n)}\leq C \|f\|_{L^p(\mathbb{R}^n)},
	\end{equation}
	when $\left(\frac{1}{p},\frac{1}{q}\right)\in\Delta_{0}(n)$.
\end{theorem}

\begin{proof}


\noindent{\bf Step 1.} We will prove that, whenever $\frac{1}{n+1}\leq\frac{1}{2}-\frac{1}{q}$, then we get  
\begin{equation}
  \|u\|_{L^q(\mathbb{R}^n)}\leq C N(f).
\end{equation}
Let $u$ be a solution of \eqref{eq:Helmmag}. Since $\nabla\cdot A\equiv 0$, we have that $(\nabla-iA)^2$ can be written as
\begin{equation}
	(\nabla-iA)^2u=\Delta u-2iA\cdot\nabla_{A}u+|A|^2u.
\end{equation}
Notice that, the electromagnetic hamiltonian can be considered as a perturbation of the free hamiltonian. Then, we have that $u$ is solution of the following equation
\begin{equation}
	\Delta u+(1\pm i\epsilon)u=f+2iA\cdot\nabla_{A}u-|A|^2u-Vu.
\end{equation}
Now we use the result from Theorem \ref{thm:triple}. If we consider the dual estimate of \eqref{eq:tripletrun}, we obtain that, if $\frac{1}{n+1}\leq\frac{1}{2}-\frac{1}{q}$ it holds
\begin{align}\label{eq:desNmag}
	\|u\|_{L^q(\mathbb{R}^n)}&=\|(\Delta+(1\pm i\epsilon))^{-1}(f+2iA\cdot\nabla_{A}u-|A|^2u-Vu)\|_{L^q(\mathbb{R}^n)}\\
	&\leq C(N(f)+N(2iA\cdot\nabla_{A}u)+N(|A|^2u)+N(Vu)).
	\nonumber
\end{align}
with $C$ independent of $\epsilon$ and $N$ defined in \eqref{eq:N}.

We are going to estimate the terms that appear on the RHS of \eqref{eq:desNmag}.
First we deal with the term $N(2iA\cdot\nabla_{A}u)$. We can split this term in two parts,
\begin{align}\label{eq:gradsplit}
	N(2iA\cdot\nabla_{A}u)^{2}&=C\sum_{j\in\Z}2^{j+1}\int_{C(j)}|A\cdot\nabla_{A}u|^{2}dx
	\\
	&=C\sum_{j\leq0}2^{j+1}\int_{C(j)}|A\cdot\nabla_{A}u|^{2}dx
	\nonumber
	\\
	&+C\sum_{j\geq0}2^{j+1}\int_{C(j)}|A\cdot\nabla_{A}u|^{2}dx.
	\nonumber
\end{align}
Let us deal both terms separately. 

\noindent First, consider $\sum_{j\geq0}2^{j+1}\int_{C(j)}|A\cdot\nabla_{A}u|^{2}dx$. Since $|A(x)|\leq C/|x|^{1+\mu}$, $\mu>0$, $|x|\geq1$, we have that this term can be bounded as
\begin{align*}
	\sum_{j\geq0}2^{j+1}\int_{C(j)}|A\cdot\nabla_{A}u|^{2}dx&\leq C\sum_{j\geq0}2^{j}\int_{C(j)}|A|^2|\nabla_{A}u|^{2}dx
	\\
	&
	\leq C\sum_{j\geq0}2^{-j}2^{-2j\mu}\int_{C(j)}|\nabla_{A}u|^{2}dx
	\nonumber
	\\
	&
	\leq C\left(\sup_{R\geq1}\frac{1}{R}\int_{1\leq|x|\leq R}|\nabla_{A}u|^{2}dx\right)\sum_{j\geq0}2^{-2j\mu}
	\nonumber
	\\
	&
	\leq C\sup_{R\geq1}\frac{1}{R}\int_{1\leq|x|\leq R}|\nabla_{A}u|^{2}dx.
	\nonumber
\end{align*}
Now, we can treat the part $\sum_{j\leq0}2^{j+1}\int_{C(j)}|A\cdot\nabla_{A}u|^{2}dx$. Let us consider $\alpha$ positive that will be fixed below, $\alpha<2/(n+1)$ (see \eqref{eq:condalpha}). 
Since $|A(x)|\leq C/|x|^\beta$, $|x|\leq1$, we have that the following holds,
\begin{align*}
\sum_{j\leq0}2^{j+1}\int_{C(j)}|A\cdot\nabla_{A}u|^{2}dx&\leq C\sum_{j\leq0}2^{j}\int_{C(j)}|A|^2|\nabla_{A}u|^{2}dx
\\
&
\leq C\sum_{j\leq0}2^j\int_{C(j)}\frac{|\nabla_{A}u|^2}{|x|^{2\beta}}dx
\nonumber
\\
&
\leq C\sum_{j\leq0}\int_{C(j)}\frac{|\nabla_{A}u|^2}{|x|^{\alpha}}|x|^{\alpha+1-2\beta}dx
\nonumber
\\
&
\leq C\int_{|x|\leq1}\frac{|\nabla_{A}u|^2}{|x|^\alpha}dx,
\nonumber
\end{align*}
provided $\alpha+1-2\beta\geq0$. This condition, together with the previous one, $\alpha<2/(n+1)$, leads us to the admissible singularity for $A$.
\begin{equation*}
2\beta-1\leq\alpha<\frac{2}{n+1}\iff\beta<\frac{n+3}{2(n+1)}.
\end{equation*}
Therefore we obtain from the separate analysis that
\begin{equation}\label{eq:Ngrad}
N(2iA\cdot\nabla_{A}u)^{2}\leq C\left(\int_{|x|\leq1}\frac{|\nabla_{A}u|^2}{|x|^\alpha}dx+\sup_{R\geq1}\frac{1}{R}\int_{1\leq|x|\leq R}|\nabla_{A}u|^{2}dx\right).
\end{equation}
Let us continue with the term $N(|A|^2u)$. As before, we can say that
\begin{align}\label{eq:A2split}
	N(|A|^2u)^{2}&=\sum_{j\in\Z}2^{j+1}\int_{C(j)}||A|^2u|^{2}dx
	\\
	&=\sum_{j\leq0}2^{j+1}\int_{C(j)}||A|^2u|^{2}dx
	\nonumber
	\\
	&+\sum_{j\geq0}2^{j+1}\int_{C(j)}||A|^2u|^{2}dx.
	\nonumber
\end{align}
The term away the origin can be treated as above, and we get
\begin{align}\label{eq:A2>}
	\sum_{j\geq0}2^{j+1}\int_{C(j)}||A|^{2}u|^{2}dx&\leq C\sum_{j\geq0}2^{j}\int_{C(j)}|A|^{4}|u|^{2}dx
	\\
	&
	\leq C\sum_{j\leq0}2^{-j}2^{-2j(1+2\mu)}\int_{C(j)}|u|^{2}dx
	\nonumber
	\\
	&
	\leq C\left(\sup_{R\geq1}\frac{1}{R}\int_{1\leq|x|\leq R}|u|^{2}dx\right)\sum_{j\geq0}2^{-2j(1+2\mu)}
	\nonumber
	\\
	&
	\leq C\sup_{R\geq1}\frac{1}{R}\int_{1\leq|x|\leq R}|u|^{2}dx.
	\nonumber
\end{align}
Let us now proceed with the part concerning to the singularity at the origin for the magnetic potential $A$. The analysis of this term will be different depending on the dimension.
We start by considering the case $n=3$. Again, since $|A(x)|\leq C/|x|^\beta$, $|x|\leq1$, it holds
\begin{align*}
	\sum_{j\leq0}2^{j+1}\int_{C(j)}||A|^2u|^{2}dx&\leq C\sum_{j\leq0}2^{j}\int_{C(j)}|A|^4|u|^{2}dx
	\nonumber
	\\
	&\leq C\sum_{j\leq0}2^{j(\alpha+3)}2^{-j(\alpha+2)}\int_{C(j)}\frac{|u|^{2}}{|x|^{4\beta}}dx
	\nonumber
	\\
	&
	=C\sum_{j\leq0}2^{j(\alpha+3-4\beta)}\frac{1}{2^{j(\alpha+2)}}\int_{C(j)}|u|^{2}dx
	\nonumber
	\\
	&
	\leq C\left(\sup_{R\leq1}\frac{1}{R^{\alpha+1}}\int_{|x|=R}|u|^2d\sigma_{R}(x)\right)\sum_{j\leq0}2^{j(\alpha+3-4\beta)}
	\nonumber
	\\
	&
	\leq C\sup_{R\leq1}\frac{1}{R^{\alpha+1}}\int_{|x|=R}|u|^2d\sigma_{R}(x),
\end{align*}
if $\alpha+3-4\beta>0$. From this , together with $\alpha<2/(n+1)$, we can conclude
\begin{equation*}
4\beta-3<\alpha<\frac{2}{n+1}\iff\beta<\frac{3n+5}{4(n+1)}.
\end{equation*}
We have that
\begin{equation*}
\frac{n+3}{2(n+1)}<\frac{3n+5}{4(n+1)},\quad\forall n>1.
\end{equation*}
Hence, if $|A(x)|\leq C/|x|^\beta$, $\beta<(n+3)/2(n+1)$, $|x|\leq1$, we get if $n=3$,
\begin{equation}\label{eq:A2<3}
\sum_{j\leq0}2^{j+1}\int_{C(j)}||A|^2u|^{2}dx\leq C\sup_{R\leq1}\frac{1}{R^{\alpha+1}}\int_{|x|=R}|u|^2d\sigma_{R}(x).
\end{equation}
Now, we treat the case $n\geq4$. We have that
\begin{align*}
\sum_{j\leq0}2^{j+1}\int_{C(j)}||A|^2u|^{2}dx&\leq C\sum_{j\leq0}2^{j}\int_{C(j)}|A|^4|u|^{2}dx
	\nonumber
	\\
	&
	\leq C\int_{|x|\leq1}\frac{|u|^{2}}{|x|^{\alpha+2}}|x|^{\alpha+3-4\beta}dx
	\nonumber
	\\
	&
	\leq C\int_{|x|\leq1}\frac{|u|^{2}}{|x|^{\alpha+2}}dx,
\end{align*}
provided $\alpha+3-4\beta\geq0$. As before, if $\beta<(n+3)/2(n+1)$, we can bound this term as
\begin{equation}\label{eq:A2<4}
\sum_{j\leq0}2^{j+1}\int_{C(j)}||A|^2u|^{2}dx\leq C\int_{|x|\leq1}\frac{|u|^{2}}{|x|^{\alpha+2}}dx.
\end{equation}
From \eqref{eq:A2split}, \eqref{eq:A2>} and \eqref{eq:A2<3}, we get for $n=3$
\begin{equation}\label{eq:NA23}
N(|A|^2u)^{2}\leq C\left(\sup_{R\leq1}\frac{1}{R^{\alpha+1}}\int_{|x|=R}|u|^2d\sigma_{R}(x)+\sup_{R\geq1}\frac{1}{R}\int_{1\leq|x|\leq R}|u|^{2}dx\right),
\end{equation}
and for $n\geq4$, from \eqref{eq:A2split}, \eqref{eq:A2>} and \eqref{eq:A2<4}, it holds
\begin{equation}\label{eq:NA24}
N(|A|^2u)^{2}\leq C\left(\int_{|x|\leq1}\frac{|u|^{2}}{|x|^{\alpha+2}}dx+\sup_{R\geq1}\frac{1}{R}\int_{1\leq|x|\leq R}|u|^{2}dx\right).
\end{equation}
Now, it remains to consider the term concerning the electric potential $V(x)$. We proceed as above by splitting the term in two parts
\begin{align}\label{eq:Vsplit}
	N(Vu)^{2}&=\sum_{j\in\Z}2^{j+1}\int_{C(j)}|Vu|^{2}dx
	\\
	&=\sum_{j\leq0}2^{j+1}\int_{C(j)}|Vu|^{2}dx
	\nonumber
	\\
	&+\sum_{j\geq0}2^{j+1}\int_{C(j)}|Vu|^{2}dx.
	\nonumber
\end{align}
Due to $V(x)\leq C/|x|^{1+\mu}$, $\mu>0$, $|x|\geq1$, the second part can be upper bounded as 
\begin{align}\label{eq:V>}
	\sum_{j\geq0}2^{j+1}\int_{C(j)}|Vu|^{2}dx&\leq C\sum_{j\geq0}2^{j}\int_{C(j)}|V|^{2}|u|^{2}dx
	\\
	&
	\leq C\sum_{j\leq0}2^{-j}2^{-2j\mu}\int_{C(j)}|u|^{2}dx
	\nonumber
	\\
	&
	\leq C\left(\sup_{R\geq1}\frac{1}{R}\int_{1\leq|x|\leq R}|u|^{2}dx\right)\sum_{j\geq0}2^{-2j\mu}
	\nonumber
	\\
	&
	\leq C\sup_{R\geq1}\frac{1}{R}\int_{1\leq|x|\leq R}|u|^{2}dx.
	\nonumber
\end{align}
The term near the origin must be treated as the magnetic potential. We consider separately the cases $n=3$ and $n\geq4$. For the $n=3$, due to the singularity assumption, namely $|V(x)|\leq C/|x|^{\gamma}$, $|x|\leq1$, we get
\begin{align*}
	\sum_{j\leq0}2^{j+1}\int_{C(j)}|Vu|^{2}dx&\leq C\sum_{j\leq0}2^{j}\int_{C(j)}|V|^{2}|u|^{2}dx
	\nonumber
	\\
	&\leq C\sum_{j\leq0}2^{j(\alpha+3)}2^{-j(\alpha+2)}\int_{C(j)}\frac{|u|^{2}}{|x|^{2\gamma}}dx
	\nonumber
	\\
	&
	=C\sum_{j\leq0}2^{j(\alpha+3-2\gamma)}\frac{1}{2^{j(\alpha+2)}}\int_{C(j)}|u|^{2}dx
	\nonumber
	\\
	&
	\leq C\left(\sup_{R\leq1}\frac{1}{R^{\alpha+1}}\int_{|x|=R}|u|^2d\sigma_{R}(x)\right)\sum_{j\leq0}2^{j(\alpha+3-2\gamma)}
	\nonumber
	\\
	&
	\leq C\sup_{R\leq1}\frac{1}{R^{\alpha+1}}\int_{|x|=R}|u|^2d\sigma_{R}(x),
\end{align*}
whenever $\alpha+3-2\gamma>0$.
This, together with $\alpha<2/(n+1)$ leads us to
\begin{equation*}
2\gamma-3<\alpha<\frac{2}{n+1}\iff\gamma<\frac{3n+5}{2(n+1)}.
\end{equation*}
Hence, we can bound this term as 
\begin{equation}\label{eq:V<3}
\sum_{j\leq0}2^{j+1}\int_{C(j)}|Vu|^{2}dx\leq C\sup_{R\leq1}\frac{1}{R^{\alpha+1}}\int_{|x|=R}|u|^2d\sigma_{R}(x).
\end{equation}
When $n\geq4$ we have 
\begin{align*}
\sum_{j\leq0}2^{j+1}\int_{C(j)}|Vu|^{2}dx&\leq C\sum_{j\leq0}2^{j}\int_{C(j)}|V|^{2}|u|^{2}dx
	\nonumber
	\\
	&
	\leq C\int_{|x|\leq1}\frac{|u|^{2}}{|x|^{\alpha+2}}|x|^{\alpha+3-2\gamma}dx
	\nonumber
	\\
	&
	\leq C\int_{|x|\leq1}\frac{|u|^{2}}{|x|^{\alpha+2}}dx,
\end{align*}
provided $\alpha+3-2\gamma\geq0$. As above we obtain
\begin{equation}\label{eq:V<4}
\sum_{j\leq0}2^{j+1}\int_{C(j)}|Vu|^{2}dx\leq C\int_{|x|\leq1}\frac{|u|^{2}}{|x|^{\alpha+2}}dx.
\end{equation}
Therefore, from \eqref{eq:Vsplit}, \eqref{eq:V>} and \eqref{eq:V<3}, we get for $n=3$
\begin{equation}\label{eq:NV3}
N(Vu)^{2}\leq C\left(\sup_{R\leq1}\frac{1}{R^{\alpha+1}}\int_{|x|=R}|u|^2d\sigma_{R}(x)+\sup_{R\geq1}\frac{1}{R}\int_{1\leq|x|\leq R}|u|^{2}dx\right), 
\end{equation}
and for $n\geq4$, from  \eqref{eq:Vsplit}, \eqref{eq:V>} and \eqref{eq:V<4}, it holds
\begin{equation}\label{eq:NV4}
N(Vu)^{2}\leq C\left(\int_{|x|\leq1}\frac{|u|^{2}}{|x|^{\alpha+2}}dx+\sup_{R\geq1}\frac{1}{R}\int_{1\leq|x|\leq R}|u|^{2}dx\right).
\end{equation}
From \eqref{eq:desNmag}, \eqref{eq:Ngrad}, \eqref{eq:NA23} and \eqref{eq:NV3}, for $n=3$ and from \eqref{eq:desNmag}, \eqref{eq:Ngrad}, \eqref{eq:NA24} and \eqref{eq:NV4} for $n\geq4$, we get that whenever $\frac{1}{n+1}\leq\frac{1}{2}-\frac{1}{q}$, we can estimate as follows
\begin{align}\label{eq:Nf}
	&\|u\|_{L^q(\mathbb{R}^n)}=\|(\Delta+(1\pm i\epsilon))^{-1}(f+2iA\cdot\nabla_{A}u-|A|^2u-Vu)\|_{L^q(\mathbb{R}^n)}\\
	&\leq C(N(f)+N(2iA\cdot\nabla_{A}u)+N(|A|^2u)+N(Vu))
	\nonumber
	\\
	&\leq CN(f)+C_{1}\left\{\left(\int_{|x|\leq1}\frac{|\nabla_{A}u|^2}{|x|^\alpha}dx\right)^{1/2}+\left(\sup_{R\geq1}\frac{1}{R}\int_{1\leq|x|\leq R}|\nabla_{A}u|^{2}dx\right)^{1/2}
	\nonumber
	\right.
	\\
	&
	\left.+\chi(n)\left(\sup_{R\leq1}\frac{1}{R^{\alpha+1}}\int_{|x|=R}|u|^2d\sigma_{R}(x)\right)^{1/2}+\left(\sup_{R\geq1}\frac{1}{R}\int_{1\leq|x|\leq R}|u|^{2}dx\right)^{1/2}
	\nonumber
	\right.
	\\
	&
	\left.+(1-\chi(n))\left(\int_{|x|\leq1}\frac{|u|^{2}}{|x|^{\alpha+2}}dx\right)^{1/2}\right\},
	\nonumber
\end{align}
where $\chi(n)$ was defined in \eqref{eq:chi}.

Now we remind the a priori estimate given by Theorem \ref{thm:apriori}, which ensures that, under the assumptions  \eqref{eq:potentialinf}-\eqref{eq:V2}, \eqref{eq:Bgret3} for $n>3$, and \eqref{eq:potentialinf}-\eqref{eq:V2}, \eqref{eq:B3} for $n=3$, there exists a constant $C$, such that the following holds
\begin{align}\label{eq:f}
\triple{u}^{2}+\triple{\nabla_{A}u}^{2}+\int_{|x|\leq1}\frac{|\nabla_{A}u|^{2}}{|x|}dx&+\sup_{R>0}\frac{1}{R^{2}}\int_{|x|=R}|u|^{2}d\sigma_{R}(x)
\\
&
+(1-\chi(n))\int_{\R^{d}}\frac{|u|^{2}}{|x|^{3}}dx\leq CN(f)^{2},
\nonumber
\end{align}
where $\chi(n)$ was defined in \eqref{eq:chi}.


\begin{remark}The constant $C$ in \eqref{eq:f} depends uniformly in $\epsilon$.
\end{remark}

\begin{remark}\label{rm:alpha}
Since $\alpha<1$, the terms containing $\alpha$ in \eqref{eq:Nf} are always smaller than the corresponding ones in \eqref{eq:f}.
\end{remark}
Then, from \eqref{eq:Nf}, \eqref{eq:f} and Remark \ref{rm:alpha}, we get
\begin{equation}\label{eq:Nff}
\|u\|_{L^q(\mathbb{R}^n)}\leq CN(f).
\end{equation}
Consequently, we get the desired estimate.

\noindent{\bf Step 2.} By applying duality to the last estimate, we get that if $\frac{1}{n+1}\leq\frac{1}{p}-\frac{1}{2}$,
\begin{equation}\label{eq:step2}
	\triple{u}\leq C \|f\|_{{L^p(\mathbb{R}^n)}}
\end{equation}
\begin{remark}The adjoint operator is the one corresponding to $\mp \epsilon$. Since we can do the same argument for both signs, all the computations are justified.
\end{remark}

\noindent{\bf Step 3.} This is the final step in the proof. As we said in the introduction the main difficulty will be to handle the first order term given by $A\cdot\nabla_{A}u$, since there are no $L^p$-$L^q$ estimates for the gradient of the solution of the free Helmholtz equation. 
In the analysis we will consider $\triple{\nabla_A u}$, being this norm controlled in $\Delta_0(n)$.

We have that, if $\left(\frac{1}{p},\frac{1}{q}\right)\in\Delta(n)$ such that $\frac{1}{n+1}\leq\frac{1}{2}-\frac{1}{q}$, from the $L^{p}$-$L^{q}$ estimates for the solution of the free equation, namely \eqref{eq:estfree}, given in Theorem \ref{thm:Helmfree}, and proceeding as in the Step 1 for the terms $2iA\cdot\nabla_{A}u$, $|A|^2u$ and $Vu$, we get
\begin{align*}
	\|u\|_{L^q(\mathbb{R}^n)}&=\|(\Delta+(1\pm i\epsilon))^{-1}(f+2iA\cdot\nabla_{A}u-|A|^2u+Vu)\|_{L^q(\mathbb{R}^n)}\\
	&\leq C\|f\|_{L^p(\mathbb{R}^n)}+C_{1}(N(2iA\cdot\nabla_{A}u)+N(|A|^2u)+N(Vu)),
	\nonumber
\end{align*}
where $C$ and $C_{1}$ do not depend on $\epsilon$.

Let us remind that, for the term $N(2iA\cdot\nabla_{A}u)$, it holds for $n\geq3$,
\begin{equation*}
N(2iA\cdot\nabla_{A}u)^{2}\leq C\left(\int_{|x|\leq1}\frac{|\nabla_{A}u|^2}{|x|^\alpha}dx+\sup_{R\geq1}\frac{1}{R}\int_{1\leq|x|\leq R}|\nabla_{A}u|^{2}dx\right).
\end{equation*}
Concerning the term $N(|A|^2u)$, we have for $n=3$
\begin{equation*}
N(|A|^2u)^{2}\leq C\left(\sup_{R\leq1}\frac{1}{R^{\alpha+1}}\int_{|x|=R}|u|^2d\sigma_{R}(x)+\sup_{R\geq1}\frac{1}{R}\int_{1\leq|x|\leq R}|u|^{2}dx\right),
\end{equation*}
and for $n\geq4$,
\begin{equation*}
N(|A|^2u)^{2}\leq C\left(\int_{|x|\leq1}\frac{|u|^{2}}{|x|^{\alpha+2}}dx+\sup_{R\geq1}\frac{1}{R}\int_{1\leq|x|\leq R}|u|^{2}dx\right).
\end{equation*}
Finally, for the term $N(Vu)$, we get for $n=3$
\begin{equation*}
N(Vu)^{2}\leq C\left(\sup_{R\leq1}\frac{1}{R^{\alpha+1}}\int_{|x|=R}|u|^2d\sigma_{R}(x)+\sup_{R\geq1}\frac{1}{R}\int_{1\leq|x|\leq R}|u|^{2}dx\right), 
\end{equation*}
and for $n\geq4$, 
\begin{equation*}
N(Vu)^{2}\leq C\left(\int_{|x|\leq1}\frac{|u|^{2}}{|x|^{\alpha+2}}dx+\sup_{R\geq1}\frac{1}{R}\int_{1\leq|x|\leq R}|u|^{2}dx\right).
\end{equation*}
Hence, we have that 
\begin{align*}
	&\|u\|_{L^q(\mathbb{R}^n)}=\|(\Delta+(1\pm i\epsilon))^{-1}(f+2iA\cdot\nabla_{A}u-|A|^2u-Vu)\|_{L^q(\mathbb{R}^n)}\\
	&\leq C\|f\|_{L^p(\mathbb{R}^n)}+C_{1}\left\{\left(\int_{|x|\leq1}\frac{|\nabla_{A}u|^2}{|x|^\alpha}dx\right)^{1/2}+\left(\sup_{R\geq1}\frac{1}{R}\int_{1\leq|x|\leq R}|\nabla_{A}u|^{2}dx\right)^{1/2}
	\nonumber
	\right.
	\\
	&
	\left.+\chi(n)\left(\sup_{R\leq1}\frac{1}{R^{\alpha+1}}\int_{|x|=R}|u|^2d\sigma_{R}(x)\right)^{1/2}+\left(\sup_{R\geq1}\frac{1}{R}\int_{1\leq|x|\leq R}|u|^{2}dx\right)^{1/2}
	\nonumber
	\right.
	\\
	&
	\left.+(1-\chi(n))\left(\int_{|x|\leq1}\frac{|u|^{2}}{|x|^{\alpha+2}}dx\right)^{1/2}\right\},
	\nonumber
\end{align*}
Finally, we conclude that if $u$ is solution of \eqref{eq:Helmmag}, by applying \eqref{eq:step2} we have that if $(\frac{1}{p},\frac{1}{q})\in\Delta_0(n)$  we conclude
\begin{align}\label{eq:grad}
	&\|u\|_{L^q(\mathbb{R}^n)}=\|(\Delta+(1\pm i\epsilon))^{-1}(f+2iA\cdot\nabla_{A}u-|A|^2u-Vu)\|_{L^q(\mathbb{R}^n)}\\
	&\leq C\|f\|_{L^p(\mathbb{R}^n)}+C_{1}\left\{\left(\int_{|x|\leq1}\frac{|\nabla_{A}u|^2}{|x|^\alpha}dx\right)^{1/2}+\left(\sup_{R\geq1}\frac{1}{R}\int_{1\leq|x|\leq R}|\nabla_{A}u|^{2}dx\right)^{1/2}
	\nonumber
	\right.
	\\
	&
	\left.+\chi(n)\left(\sup_{R\leq1}\frac{1}{R^{\alpha+1}}\int_{|x|=R}|u|^2d\sigma_{R}(x)\right)^{1/2}+(1-\chi(n))\left(\int_{|x|\leq1}\frac{|u|^{2}}{|x|^{\alpha+2}}dx\right)^{1/2}\right\}.
	\nonumber
\end{align}
We have to bound the remaining factors in \eqref{eq:grad} in terms of the $L^{p}$ norm of $f$. This will be done in two parts. During the discussion we will distinguish between the terms related with the behaviour of the solution $u$ near the origin and those ones related with the behaviour at infinity. Therefore, we consider separately the cases $R\leq1$ and $R\geq1$. Let us first proceed with the case $R\leq 1$.

Consider a radial function $\varphi=\varphi(|x|)\in C_{0}^\infty(\R^n)$. By multiplying Helmholtz equation \eqref{eq:Helmmag} by $\varphi\bar{u}$ in the $L^2$-sense and taking the resulting real parts it gives the identity
\begin{align}\label{eq:identity}
-&\int_{\R^{n}}\varphi|\nabla_{A}u|^2dx+\frac{1}{2}\int_{\R^{n}}\Delta\varphi|u|^2dx+\int_{\R^{n}}\varphi V|u|^2dx+\int_{\R^{n}}\varphi|u|^2dx\\
&=\Re\int_{\R^{n}} f\varphi\bar{u}dx.
\nonumber
\end{align}
Then, we choose the appropriate multiplier $\varphi$. Consider, for $R\leq1$, $\varphi_{R}=\eta\psi_{R}$, where $\eta\in C^\infty_{0}(\R^n)$ and $\psi_{R}$ are defined respectively by
\begin{equation}
\eta(x)=
\begin{cases}
1\quad\text{if}\quad|x|\leq1,\\
0\quad\text{if}\quad|x|\geq2,
\end{cases}
\end{equation}
and
\begin{equation}
\psi_{R}(x)=
\begin{cases}
-\frac{1}{R^{\alpha}}\quad\text{if}\quad|x|\leq R,\\
-\frac{1}{|x|^{\alpha}}\quad\text{if}\quad R<|x|\leq 1,\\
0\qquad\text{if}\quad|x|>2.
\end{cases}
\end{equation}
where $\alpha<2/(n+1)$.

Notice that $\varphi_{R}$ is continuous. Since $R\leq1$ and this multiplier is chosen in order to control the terms near the origin in \eqref{eq:grad}, the interesting region is $|x|\leq1$, and it is easy to compute that $\Delta\varphi_{R}$ is given by
\begin{equation*}
\Delta\varphi_{R}(x)=\frac{\alpha}{R^{\alpha+1}}\delta_{|x|=R}+\frac{\alpha(n-(\alpha+2))}{|x|^{\alpha+2}}\chi_{\{R<|x|\leq1\}}(x).
\end{equation*}
By inserting $\varphi_{R}$ in \eqref{eq:identity}, we get
\begin{align}\label{eq:ineqR<1}
	&\frac{1}{R^{\alpha}}\int_{B_{R}}|\nabla_{A}u|^{2}dx+\int_{R<|x|\leq1}\frac{|\nabla_{A}u|^{2}}{|x|^{\alpha}}dx+\int_{1<|x|\leq2}|\varphi_{R}||\nabla_{A}u|^{2}dx\\
	&
	\frac{\alpha}{2R^{\alpha+1}}\int_{|x|=R}|u|^{2}d\sigma_{R}(x)+\frac{\alpha(n-(\alpha+2))}{2}\int_{R<|x|\leq1}\frac{|u|^{2}}{|x|^{\alpha+2}}dx
	\nonumber
	\\
	&
	+\frac{1}{2}\int_{1<|x|\leq2}\Delta\varphi_{R}|u|^{2}dx\leq|\int_{\R^{n}}\varphi_{R}f\bar{u}dx|+\frac{1}{R^{\alpha}}\int_{B_{R}}|V||u|^{2}dx
	\nonumber
	\\
	&
	+\int_{R<|x|\leq1}\frac{|V||u|^{2}}{|x|^{\alpha}}dx+\int_{1<|x|\leq2}|\varphi_{R}||V||u|^{2}dx+\frac{1}{R^{\alpha}}\int_{B_{R}}|u|^{2}dx
	\nonumber
	\\
	&
	+\int_{R<|x|\leq1}\frac{|u|^{2}}{|x|^{\alpha}}dx+\int_{1<|x|\leq2}|\varphi_{R}||u|^{2}dx.
	\nonumber
\end{align}
We are going to bound the terms appearing in the RHS of \eqref{eq:ineqR<1}. First, we consider $\frac{1}{R^{\alpha}}\int_{B_{R}}|V||u|^{2}dx$. Recall that $|V(x)|\leq C/|x|^{\gamma}$, $\gamma<(3n+5)/2(n+1)$, $|x|\leq1$. Therefore, we can proceed for $n=3$ as follows
\begin{align*}
	\frac{1}{R^{\alpha}}\int_{B_{R}}|V||u|^{2}dx&\leq\int_{B_{R}}\frac{|V||u|^{2}}{|x|^{\alpha}}dx=\int_{0}^{R}\frac{d\rho}{\rho^{\alpha}}\int_{|x|=\rho}|V||u|^{2}d\sigma_{\rho}(x)
	\\
	&
	\leq\left(\int_{0}^{1}\sup_{|x|=\rho}|x||V(x)|d\rho\right)\sup_{R\leq1}\frac{1}{R^{\alpha+1}}\int_{|x|=R}|u|^2d\sigma_{R}(x).
	\nonumber
\end{align*}
Let us denote by
\begin{equation}\label{eq:norm3}
	\||x|V\|_{L^{1}_{r<1}L^{\infty}(S_{r})}=\int_{0}^{1}\sup_{|x|=\rho}|x||V(x)|d\rho.
\end{equation}
This norm is obviously finite under the assumptions for $V(x)$ in \eqref{eq:H1},
\begin{equation*}
	\||x|V\|_{L^{1}_{r<1}L^{\infty}(S_{r})}=\int_{0}^{1}\sup_{|x|=\rho}|x||V(x)|d\rho\leq C\int_{0}^{1}\frac{d\rho}{\rho^{\gamma-1}}.
\end{equation*}
The last integral is finite whenever $\gamma<2$. We have that 
\begin{equation*}
\frac{3n+5}{2(n+1)}<2\iff n>1.
\end{equation*}
Consequently, we can ensure that for $n=3$
\begin{equation*}
	\frac{1}{R^{\alpha}}\int_{B_{R}}|V||u|^{2}\leq\||x|V\|_{L^{1}_{r<1}L^{\infty}(S_{r})}\sup_{R\leq1}\frac{1}{R^{\alpha+1}}\int_{|x|=R}|u|^2d\sigma_{R}(x).
\end{equation*}
Now, we treat the term $\int_{R<|x|\leq1}\frac{|V||u|^{2}}{|x|^{\alpha}}dx$. Proceeding as above, we get
\begin{align*}
	\int_{R<|x|\leq1}\frac{|V||u|^{2}}{|x|^{\alpha}}dx&=\int_{R}^{1}\frac{d\rho}{\rho^{\alpha}}\int_{|x|=\rho}|V||u|^{2}d\sigma_{\rho}(x)
	\\
	&\leq\left(\int_{0}^{1}\sup_{|x|=\rho}|x||V(x)|d\rho\right)\sup_{R\leq1}\frac{1}{R^{\alpha+1}}\int_{|x|=R}|u|^2d\sigma_{R}(x).
	\nonumber
\end{align*}
Hence,
\begin{equation*}
	\int_{R<|x|\leq1}\frac{|V||u|^{2}}{|x|^{\alpha}}dx\leq\||x|V\|_{L^{1}_{r<1}L^{\infty}(S_{r})}\sup_{R\leq1}\frac{1}{R^{\alpha+1}}\int_{|x|=R}|u|^2d\sigma_{R}(x).
\end{equation*}
Let us now consider the case $n\geq4$. We have that
\begin{align*}
	\frac{1}{R^{\alpha}}\int_{B_{R}}|V||u|^{2}dx&\leq\int_{B_{R}}\frac{|V||u|^{2}}{|x|^{\alpha}}dx
	\\
	&
	\leq\||x|^{2}V\|_{L^{\infty}(|x|\leq1)}\int_{|x|\leq1}\frac{|u|^{2}}{|x|^{\alpha+2}}dx
	\nonumber
\end{align*}
Moreover
\begin{equation*}
	\||x|^{2}V\|_{L^{\infty}(|x|\leq1)}=\sup_{|x|\leq1}|x||V(x)|\leq C\sup_{|x|\leq1}|x|^{2-\gamma}=C
\end{equation*}
provided $\gamma\leq2$. We have that this condition holds, since $\gamma<(3n+5)/2(n+1)$.

Similarly, we get
\begin{equation*}
	\int_{R<|x|\leq1}\frac{|V||u|^{2}}{|x|^{\alpha}}dx\leq\||x|^{2}V\|_{L^{\infty}(|x|\leq1)}\int_{|x|\leq1}\frac{|u|^{2}}{|x|^{\alpha+2}}dx.
\end{equation*}
Now, we consider $\int_{R<|x|\leq1}\frac{|u|^{2}}{|x|^{\alpha}}dx$. This can be bounded as follows. Let $J\in\Z$, $j\leq 0$ such that $2^{J}<R<2^{J+1}$. We get
\begin{align*}
	\int_{R<|x|\leq1}\frac{|u|^{2}}{|x|^{\alpha}}dx&=\sum_{J<j\leq 0}\int_{C(j)}\frac{|u|^{2}}{|x|^{\alpha}}dx\leq \sum_{j\leq0}\int_{2^j}^{2^{j+1}}\frac{|u|^{2}}{2^{j\alpha}}dx
	\\
	&
	\leq C\sup_{R\leq1}\frac{1}{R^{\alpha+1}}\int_{|x|=R}|u|^{2}d\sigma_R(x).
	\nonumber
\end{align*}
In the same way, we obtain
\begin{equation*}
\frac{1}{R^\alpha}\int_{B_R}|u|^2dx\leq C\sup_{R\leq1}\frac{1}{R^{\alpha+1}}\int_{|x|=R}|u|^{2}d\sigma_R(x)
\end{equation*}
We also have that
\begin{equation*}
	\int_{1<|x|\leq2}|\varphi_{R}||V||u|^{2}dx+\int_{1<|x|\leq2}|\varphi_{R}||u|^{2}dx\leq\sup_{R\geq1}\frac{1}{R}\int_{1<|x|\leq R}|u|^{2}dx.
\end{equation*}
The last term is the one which contains $f$. Let us apply H\"older inequality to this term, obtaining
\begin{equation}\label{eq:hol}
	\int_{\mathbb{R}^n}|\varphi_{R}f\bar{u}|dx\leq\|f\|_{L^{p}}\|u\|_{L^{q}}\|\varphi_{R}\|_{L^r},\\
\end{equation}
whenever $1=1/p+1/q+1/r$.

We are going to consider the following subregion of $\Delta(n)$, for $n\geq3$, which will be crucial during the discussion. It is defined by
\begin{equation*}
	\Delta_{0}(n)=\left\{\left(\frac{1}{p},\frac{1}{q}\right)\in\Delta(n):\frac{1}{n+1}\leq \frac{1}{p}-\frac{1}{2},\frac{1}{n+1}\leq \frac{1}{2}-\frac{1}{q}\right\}.
\end{equation*}
$\Delta_{0}(n)$ is the solid triangle determined by the points $Q$, $Q'$ and $Q''$ (See Figure \ref{fig:deltan02S}).
\begin{figure}[!htb]
\begin{center}
\includegraphics[]{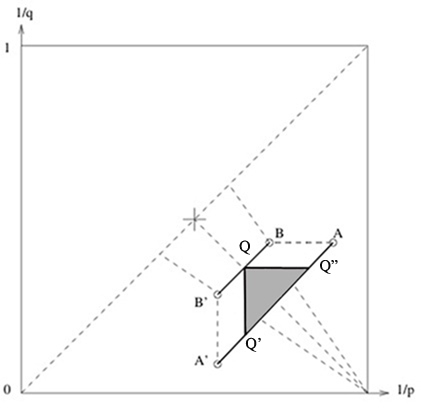}
\caption{$\Delta_{0}(n)$, $n\geq3$.}
\label{fig:deltan02S}
\end{center}
\end{figure}
Let us consider the point $Q'\in\Delta(n)$ and apply H\"older inequality when $p$ and $q$ are the coordinates of $Q'$. The point $Q'$ satisfies that
\begin{equation*}
	\frac{1}{p}-\frac{1}{2}=\frac{1}{n+1},\qquad\frac{1}{2}-\frac{1}{q}=\frac{n+2}{n(n+1)}.
\end{equation*}
Therefore, the corresponding $r$ in \eqref{eq:hol} is given by
\begin{equation*}
	\frac{1}{r}=\left(\frac{1}{2}-\frac{1}{q}\right)-\left(\frac{1}{p}-\frac{1}{2}\right)=\frac{2}{n(n+1)}.
\end{equation*}
Then, we have to compute $\|\varphi_{R}\|_{L^r}$, with $r$ given above, under the assumption $R\leq1$. If $n-\alpha r>0$, we get
\begin{equation}\label{eq:condalpha}
	\|\varphi_{R}\|_{L^r}^{r}\leq \frac{C}{n-\alpha r}R^{n-\alpha r}.
\end{equation}
So, in order to keep under control the factor $n-\alpha r$ in the denominator, we take $\alpha<2/(n+1)$. This choice of $\alpha$ justifies the allowed singularity for the potentials $A$ and $V$ given by the (H1) assumption (See \eqref{eq:H1}). With this choice we have that
\begin{equation*}
 n-\alpha r>0.
\end{equation*}
Therefore, we have that, for the point $Q'$, the $L^{r}$ norm of $\varphi_{R}$ is bounded from above
\begin{equation*}
	\|\varphi_{R}\|_{L^r}^{r}\leq C.
\end{equation*}
Now, after taking supremum in \eqref{eq:ineqR<1}, we get for the the point $Q'$ in the $3$-D case
\begin{align*}
&\sup_{R\leq1}\frac{1}{R^{\alpha}}\int_{B_{R}}|\nabla_{A}u|^{2}dx+\int_{|x|\leq1}\frac{|\nabla_{A}u|^{2}}{|x|^{\alpha}}dx+\frac{\alpha}{2}\sup_{R\leq1}\frac{1}{R^{\alpha+1}}\int_{|x|=R}|u|^{2}d\sigma_{R}(x)\\
&\leq C\|f\|_{L^{p}}\|u\|_{L^{q}}+C_1\||x|V\|_{L^{1}_{r<1}L^{\infty}(S_{r})}\sup_{R\leq1}\frac{1}{R^{\alpha+1}}\int_{|x|=R}|u|^2d\sigma_{R}(x)
\nonumber
\\
&
+C_{2}\sup_{R\geq1}\frac{1}{R}\int_{1<|x|\leq R}|u|^{2}dx
\nonumber
\end{align*}
Therefore by making $\||x|V\|_{L^{1}_{r<1}L^{\infty}(S_{r})}$ small enough, we
finally, we obtain for $n=3$
\begin{align*}
&\sup_{R\leq1}\frac{1}{R^{\alpha}}\int_{B_{R}}|\nabla_{A}u|^{2}dx+\int_{|x|\leq1}\frac{|\nabla_{A}u|^{2}}{|x|^{\alpha}}dx+\sup_{R\leq1}\frac{1}{R^{\alpha+1}}\int_{|x|=R}|u|^{2}d\sigma_{R}(x)\\
&\leq C\left(\|f\|_{L^{p}}\|u\|_{L^{q}}+\sup_{R\geq1}\frac{1}{R}\int_{1<|x|\leq R}|u|^{2}dx\right).
\nonumber	
\end{align*}
where $p$ and $q$ are the coordinates of the point $Q'$.

Similarly, for $n\geq4$ and observing that, since $\alpha<1$ we have that
\begin{equation*}
\frac{\alpha(n-(\alpha+2))}{2}>\frac{\alpha(n-3)}{2},
\end{equation*}
we get, after taking supremum in \eqref{eq:ineqR<1},
\begin{align*}
&\sup_{R\leq1}\frac{1}{R^{\alpha}}\int_{B_{R}}|\nabla_{A}u|^{2}dx+\int_{|x|\leq1}\frac{|\nabla_{A}u|^{2}}{|x|^{\alpha}}dx+\frac{\alpha}{2}\sup_{R\leq1}\frac{1}{R^{\alpha+1}}\int_{|x|=R}|u|^{2}d\sigma_{R}(x)\\
&+\frac{\alpha(n-3)}{2}\int_{|x|\leq1}\frac{|u|^{2}}{|x|^{\alpha+2}}dx\leq C\|f\|_{L^{p}}\|u\|_{L^{q}}+C_1\||x|^{2}V\|_{L^{\infty}(|x|\leq1)}\int_{|x|\leq1}\frac{|u|^{2}}{|x|^{\alpha+2}}dx
\nonumber
\\
&
+C_{2}\sup_{R\geq1}\frac{1}{R}\int_{1<|x|\leq R}|u|^{2}dx.
\nonumber	
\end{align*}
As before if we take $\||x|^{2}V\|_{L^{\infty}(|x|\leq1)}$ sufficiently small ,
we get for $n\geq4$
\begin{align*}
&\sup_{R\leq1}\frac{1}{R^{\alpha}}\int_{B_{R}}|\nabla_{A}u|^{2}dx+\int_{|x|\leq1}\frac{|\nabla_{A}u|^{2}}{|x|^{\alpha}}dx+\sup_{R\leq1}\frac{1}{R^{\alpha+1}}\int_{|x|=R}|u|^{2}d\sigma_{R}(x)
\\
&
+\int_{|x|\leq1}\frac{|u|^{2}}{|x|^{\alpha+2}}dx\leq C\left(\|f\|_{L^{p}}\|u\|_{L^{q}}+\sup_{R\geq1}\frac{1}{R}\int_{1<|x|\leq R}|u|^{2}dx\right),
\nonumber	
\end{align*}
where $p$ and $q$ are the coordinates of the point $Q'$.


Now, we are going to deal with the case $R\geq1$. Let us consider $\varphi_{R}\in C^{\infty}_{0}(\R^n)$ given by
\begin{equation*}
	\varphi_{R}=
	\begin{cases}
		0\qquad\text{if}\qquad|x|<\frac{1}{2},\\
		-\frac{1}{R}\quad\text{if}\quad1<|x|\leq R,\\
		0\qquad\text{if}\qquad|x|>2R.
	\end{cases}
\end{equation*}
By inserting this function in \eqref{eq:identity}, we get
\begin{align}\label{eq:ineqR>1}
	&\int_{\frac{1}{2}<|x|\leq1}|\varphi_{R}||\nabla_{A}u|^{2}dx+\frac{1}{R}\int_{1<|x|\leq R}|\nabla_{A}u|^{2}dx+\int_{R<|x|\leq2R}|\varphi_{R}||\nabla_{A}u|^{2}dx
	\\
	&
	+\frac{1}{2}\int_{|x|\leq2R}\Delta\varphi_{R}|u|^{2}dx\leq|\int_{\R^{n}}\varphi_{R}f\bar{u}dx|+\int_{\frac{1}{2}<|x|\leq1}|\varphi_{R}||V||u|^{2}dx
	\nonumber
	\\
	&
	+\frac{1}{R}\int_{1<|x|\leq R}|V||u|^{2}dx+\int_{R<|x|\leq2R}|\varphi_{R}||V||u|^{2}dx+\int_{\frac{1}{2}<|x|\leq1}|\varphi_{R}||u|^{2}dx
	\nonumber
	\\
	&
	+\frac{1}{R}\int_{1<|x|\leq R}|u|^{2}dx+\int_{R<|x|\leq2R}|\varphi_{R}||u|^{2}dx.
	\nonumber
\end{align}
Some terms on the LHS of \eqref{eq:ineqR>1} are positive, so we discard them, obtaining the following inequality
\begin{align}\label{eq:1}
	&\frac{1}{R}\int_{1<|x|\leq R}|\nabla_{A}u|^{2}dx\leq|\int_{\R^{n}}\varphi_{R}f\bar{u}dx|+\int_{\frac{1}{2}<|x|\leq1}|\varphi_{R}||V||u|^{2}dx
	\nonumber
	\\
	&
	+\frac{1}{R}\int_{1<|x|\leq R}|V||u|^{2}dx+\int_{R<|x|\leq2R}|\varphi_{R}||V||u|^{2}dx+\int_{\frac{1}{2}<|x|\leq1}|\varphi_{R}||u|^{2}dx
	\nonumber
	\\
	&+\frac{1}{R}\int_{1<|x|\leq R}|u|^{2}dx+\int_{R<|x|\leq2R}|\varphi_{R}||u|^{2}dx.
	\nonumber
\end{align}
We treat the terms appearing on the RHS of the previous inequality. Since $|V(x)|\leq C/|x|^{1+\mu}$, $|x|\geq1$, we can proceed as follows
\begin{align}
	\frac{1}{R}\int_{1<|x|\leq R}|V||u|^{2}dx&\leq\int_{1<|x|\leq R}\frac{|V||u|^{2}}{|x|}dx
	\\
	&
	\leq C\sum_{j\geq0}2^{-j(1+\mu)}2^{-j}\int_{C(j)}|u|^{2}dx
	\nonumber
	\\
	&
	\leq C\left(\sup_{R\geq1}\frac{1}{R}\int_{1\leq|x|\leq R}|u|^{2}dx\right)\sum_{j\geq0}2^{-j(1+\mu)}
	\nonumber
	\\
	&
	\leq C\sup_{R\geq1}\frac{1}{R}\int_{1\leq|x|\leq R}|u|^{2}dx.
	\nonumber
\end{align}
Similarly, we have
\begin{align*}
	\int_{R<|x|\leq2R}|\varphi_{R}||V||u|^{2}dx&\leq\frac{1}{R}\int_{1<|x|\leq2R}|V||u|^{2}dx
	\\
	&
	\leq C\sup_{R\geq1}\frac{1}{R}\int_{1\leq|x|\leq R}|u|^{2}dx.
	\nonumber
\end{align*}
Moreover
\begin{equation*}
	\int_{\frac{1}{2}<|x|\leq1}|\varphi_{R}||V||u|^{2}dx\leq C\sup_{R\geq1}\frac{1}{R}\int_{1\leq|x|\leq R}|u|^{2}dx,
\end{equation*}
\begin{equation*}
	\int_{\frac{1}{2}<|x|\leq1}|\varphi_{R}||u|^{2}dx\leq C\sup_{R\geq1}\frac{1}{R}\int_{1\leq|x|\leq R}|u|^{2}dx,
\end{equation*}
\begin{equation*}
	\frac{1}{R}\int_{1<|x|\leq R}|u|^{2}dx\leq \sup_{R\geq1}\frac{1}{R}\int_{1\leq|x|\leq R}|u|^{2}dx,
\end{equation*}
and
\begin{align}\label{eq:bR>1ul}
	\int_{R<|x|\leq2R}|\varphi_{R}||u|^{2}dx&\leq\frac{1}{R}\int_{1<|x|\leq2R}|u|^{2}dx
	\\
	&
	\leq C\sup_{R\geq1}\frac{1}{R}\int_{1\leq|x|\leq R}|u|^{2}dx.
	\nonumber
\end{align}
Hence, after taking supremum in the LHS of \eqref{eq:ineqR>1}, we obtain
\begin{equation*}
	\sup_{R\geq1}\frac{1}{R}\int_{1<|x|\leq R}|\nabla_{A}u|^{2}dx\leq|\int_{\R^{n}}\varphi_{R}f\bar{u}dx|+C\sup_{R\geq1}\frac{1}{R}\int_{1\leq|x|\leq R}|u|^{2}dx.
\end{equation*}
As when we discussed the case $R\leq1$, for the term containing $f$, we apply H\"older inequality and obtain
\begin{equation}\label{eq:hol2}
	\int_{\mathbb{R}^n}|\varphi_{R}f\bar{u}|dx\leq\|f\|_{L^{p}}\|u\|_{L^{q}}\|\varphi_{R}\|_{L^r},\\
\end{equation}
whenever $1=1/p+1/q+1/r$.

Therefore, we have to compute $\|\varphi_{R}\|_{L^r}$ under the assumption $R\geq1$. We get the following
\begin{equation*}
	\|\varphi_{R}\|_{L^r}^{r}\leq CR^{n-r}.
\end{equation*}
This term is bounded for $R\geq1$ if
\begin{equation}\label{eq:conditions}
\frac{n}{r}\leq1,\qquad1=\frac{1}{p}+\frac{1}{q}+\frac{1}{r}.
\end{equation}
So, we have to find the region for $p$ and $q$, for $(\frac{1}{p},\frac{1}{q})\in\Delta_{0}(n)$, where \eqref{eq:conditions} are satisfied.

We have that \eqref{eq:conditions} holds if $(\frac{1}{p},\frac{1}{q})\in\Delta_{0}^{-}(n)$, where $\frac{\Delta_{0}(n)}{2}$ is given by
\begin{equation*}
	\Delta_{0}^{-}(n)=\left\{\left(\frac{1}{p},\frac{1}{q}\right)\in\Delta_{0}(n):\frac{1}{q}\leq 1-\frac{1}{p}\right\}.
\end{equation*}
The region $\Delta_{0}^{-}(n)$ is determined by the points $\left(\frac{1}{p},\frac{1}{q}\right)\in\Delta_{0}(n)$ under the diagonal $1/q=1-1/p$, including the points $\left(\frac{1}{p},\frac{1}{q}\right)$ in the duality line (See Figure \ref{fig:deltan0med2}).

\begin{figure}[!htb]
\begin{center}
\includegraphics[]{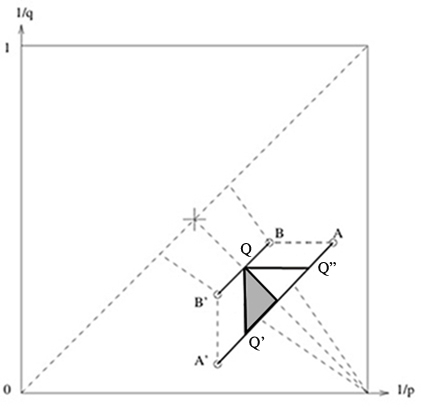}
\caption{$\Delta_{0}^{-}(n)$, $n\geq3$.}
\label{fig:deltan0med2}
\end{center}
\end{figure}
Hence, if $(\frac{1}{p},\frac{1}{q})\in\Delta_{0}^{-}(n)$, we can apply H\"older inequality and we obtain
\begin{equation}\label{eq:boundR>1}
	\sup_{R\geq1}\frac{1}{R}\int_{1<|x|\leq R}|\nabla_{A}u|^{2}dx\leq C\|f\|_{L^{p}}\|u\|_{L^{q}}+C_{1}\sup_{R\geq1}\frac{1}{R}\int_{1\leq|x|\leq R}|u|^{2}dx.
\end{equation}
Recall that we were able to apply H\"older inequality for the case $R\leq1$ if $p$ and $q$ were the coordinates of the point $Q'$. We are going to remind the results we obtained. For $n=3$, we got the following
\begin{align}\label{eq:boundR<13}
&\sup_{R\leq1}\frac{1}{R^{\alpha}}\int_{B_{R}}|\nabla_{A}u|^{2}dx+\int_{|x|\leq1}\frac{|\nabla_{A}u|^{2}}{|x|^{\alpha}}dx+\sup_{R\leq1}\frac{1}{R^{\alpha+1}}\int_{|x|=R}|u|^{2}d\sigma_{R}(x)\\
&\leq C\left(\|f\|_{L^{p}}\|u\|_{L^{q}}+\sup_{R\geq1}\frac{1}{R}\int_{1<|x|\leq R}|u|^{2}dx\right).
\nonumber	
\end{align}
and for $n\geq4$
\begin{align}\label{eq:boundR<14}
&\sup_{R\leq1}\frac{1}{R^{\alpha}}\int_{B_{R}}|\nabla_{A}u|^{2}dx+\int_{|x|\leq1}\frac{|\nabla_{A}u|^{2}}{|x|^{\alpha}}dx+\sup_{R\leq1}\frac{1}{R^{\alpha+1}}\int_{|x|=R}|u|^{2}d\sigma_{R}(x)
\\
&
+\int_{|x|\leq1}\frac{|u|^{2}}{|x|^{\alpha+2}}dx\leq C\left(\|f\|_{L^{p}}\|u\|_{L^{q}}+\sup_{R\geq1}\frac{1}{R}\int_{1<|x|\leq R}|u|^{2}dx\right).
\nonumber	
\end{align}
Now, we are going to apply all the obtained results in order to get the desired estimate. From the analysis for the cases $R\leq1$ and $R\geq1$ we have that, for $n\geq3$, from \eqref{eq:boundR>1}, \eqref{eq:boundR<13} in the $3$-D case, and from \eqref{eq:boundR>1},\eqref{eq:boundR<14} for higher dimensions, we can ensure that for the point $Q'$ it holds
\begin{equation*}
	\|u\|_{L^q(\mathbb{R}^n)}
	\leq C\|f\|_{L^p(\R^n)}+C_1\left(\|f\|_{L^{p}}\|u\|_{L^{q}}+\sup_{R\geq1}\frac{1}{R}\int_{1<|x|\leq R}|u|^{2}dx\right)^{\frac{1}{2}}.
	\nonumber
\end{equation*}
Then, we remind \eqref{eq:step2}. It verifies that, whenever $\frac{1}{n+1}\leq\frac{1}{p}-\frac{1}{2}$ we have
\begin{equation*}
	\triple{u}\leq C \|f\|_{{L^p(\mathbb{R}^n)}}.
\end{equation*}
By using this inequality,  we have that, for the point $Q'$, it holds
\begin{align*}
	\|u\|_{L^q(\mathbb{R}^n)}&=\|(\Delta+(1\pm i\epsilon))^{-1}(f+2iA\cdot\nabla_{A}u-|A|^2u-Vu)\|_{L^q(\mathbb{R}^n)}
	\nonumber
	\\
	&
	\leq C\|f\|_{L^{p}}+\lambda\|u\|_{L^{q}}+C(\lambda)\|f\|_{L^{p}}.
\end{align*}
for $\lambda, C(\lambda)>0$. We choose $\lambda$ sufficiently small and conclude the desired bound for the point $Q'$
\begin{equation}\label{eq:finalestimate}
		\|u\|_{L^q(\mathbb{R}^n)}\leq C \|f\|_{L^p(\mathbb{R}^n)}.
\end{equation}
We also have that \eqref{eq:finalestimate} is true for the points in the segment $QQ'$, (see Figure \ref{fig:deltan0med2}), since the coordinate $p$ is the same for all that points, and the boundedness of the solution when the coordinates $p$ and $q$ belong to points in the segment $QQ'$ is derived from the corresponding one for $Q'$. 

Now, by duality, the same holds for the points in the segment $QQ''$ and finally, by an interpolation argument, we have that, if $(\frac{1}{p},\frac{1}{q})\in\Delta_{0}(n)$, we get the desired bound 
\begin{equation*}
		\|u\|_{L^q(\mathbb{R}^n)}\leq C \|f\|_{L^p(\mathbb{R}^n)}.
\end{equation*}
The proof is complete.
\end{proof}
We have that, due to the presence of the first order term associated to the gradient, the region where we prove the estimates for the solution of the electromagnetic Helmholtz equation is smaller than the one valid in the free case. Therefore it is natural to wonder if, whenever $A\equiv0$ (i.e., we deal with the Helmholtz equation with electric potential), this region can be extended to the whole $\Delta(n)$. Therefore, we will pass to consider the case of Helmholtz equation with electric potential $V$. Logically, outside $\Delta_{0}(n)$, the assumed decay for $V$ will not be enough in order to get the estimates.

In the following result we extend the boundedness of the solution to the whole $\Delta(n)$. We obtain the next result.

\begin{theorem}
\label{thm:extension}

Let $u\in H^1(\R^n)$ be a solution of 
	\begin{equation}\label{eq:Helmelec}
		\Delta u+(1\pm i\epsilon)u+V(x)u=f,\quad \text{in}\quad \mathbb{R}^n,\quad n\geq3,\quad\epsilon\neq0.
	\end{equation}
	Let $V$ satisfies assumptions  \eqref{eq:potentialinf}-\eqref{eq:V2} for $n\geq3$ in Theorem \ref{thm:apriori}, and suppose that there exist positive constants $C,\gamma,\delta$ such that
	\begin{equation}
		|V(x)|\leq 
		\begin{cases}
		\frac{C}{|x|^{\gamma}},\quad\text{if}\quad|x|\leq1,\\
		\frac{C}{|x|^{\delta}},\quad\text{if}\quad|x|\geq1
		\end{cases}
	\end{equation}
	where $\gamma$ and $\delta$ satisfies for $\left(\frac{1}{p},\frac{1}{q}\right)\in\Delta(n)\setminus\Delta_{0}(n)$,
	\begin{equation}
	\gamma <
	\begin{cases}
	\frac{1}{2}+n\left\{\frac{2}{n+1}-(\frac{1}{2}-\frac{1}{q})\right\},\quad\frac{1}{2n}<\frac{1}{2}-\frac{1}{q}<\frac{1}{n+1},\\
	\frac{1}{2}+n\left\{\frac{2}{n+1}-(\frac{1}{p}-\frac{1}{2})\right\},\quad\frac{1}{2n}<\frac{1}{p}-\frac{1}{2}<\frac{1}{n+1},
	\end{cases}
	\end{equation}
	\begin{equation}
	\delta >
	\begin{cases}
	\frac{1}{2}+n\left\{\frac{2}{n+1}-(\frac{1}{2}-\frac{1}{q})\right\},\quad\frac{1}{2n}<\frac{1}{2}-\frac{1}{q}<\frac{1}{n+1},\\
	\frac{1}{2}+n\left\{\frac{2}{n+1}-(\frac{1}{p}-\frac{1}{2})\right\},\quad\frac{1}{2n}<\frac{1}{p}-\frac{1}{2}<\frac{1}{n+1}.
	\end{cases}
	\end{equation}
	Then, there exists a constant C, independent of $\epsilon$, such that 
	\begin{equation}
		\|u\|_{L^q(\mathbb{R}^n)}\leq C \|f\|_{L^p(\mathbb{R}^n)}.
	\end{equation}
\end{theorem}
\begin{proof}
As for the proof of Theorem \ref{thm:elecmag}, this will be divided in three steps. The first two steps are the same, so we will skip them. The main difference appears in Step 3.
\begin{remark}\label{rm:proof}
As we have seen in Theorem \ref{thm:elecmag}, for $\left(\frac{1}{p},\frac{1}{q}\right)\in\Delta_{0}(n)$, the short-range decay assumptions for the electric potential $V$ and the magnetic potential $A$ are sufficient in order to prove the $L^{p}$-$L^{q}$ estimates. As we will see, outside this region, more decay for the electric potential $V$ potential is needed. 
\end{remark}

\noindent{\bf Step 3.} Now let us consider $\left(\frac{1}{p},\frac{1}{q}\right)\in\Delta(n)$ such that $\frac{1}{2n}<\frac{1}{2}-\frac{1}{q}<\frac{1}{n+1}$. Then, from Theorem \ref{thm:Helmfree} and observing that, since the dual estimate of \eqref{eq:tripletrun} in Theorem \ref{thm:triple} can not be applied for the perturbative term $Vu$ because we are outside the allowed range for $q$, we conclude that the $L^q$ norm of the solution of the equation \eqref{eq:Helmelec} can be bounded as follows 
\begin{align}\label{eq:Vu}
	\|u\|_{L^q(\mathbb{R}^n)}&=\|(\Delta+(1\pm i\epsilon))^{-1}(f-Vu)\|_{L^q(\mathbb{R}^n)}\\
	&
	\leq C(\|f\|_{{L^p(\mathbb{R}^n)}}+\|Vu\|_{{L^{p_{1}}(\mathbb{R}^n)}}).
	\nonumber
\end{align}
where $C$ does not depend on $\epsilon$.

Here $p_{1}$ is given by
\begin{equation*}
	\frac{1}{p_{1}}-\frac{1}{q}=\frac{2}{n+1},\quad\frac{1}{2n}<\frac{1}{2}-\frac{1}{q}<\frac{1}{n+1}.
\end{equation*}
We have taken the point $p_{1}$ being in the line $1/p_{1}-1/q=2/(n+1)$ since we want to require the smallest decay at infinity for $V$.

In order to give the necessary conditions that must satisfy the electric potential, namely, the admissible singularity at the origin and the decay at infinity, we analyze separately the $L^{p_{1}}$ norm of $Vu$. 

We have that 
\begin{equation*}
	\int_{|x|\leq1}|Vu|^{p_{1}}dx=\sum_{j<0}\int_{C(j)}|Vu|^{p_{1}}dx
\end{equation*}
where, $C(j)$ was the annulus given by
\begin{equation*}
C(j)=\{x\in\mathbb{R}^n:2^j\leq|x|\leq2^{j+1}\},
\end{equation*}
Therefore, if $|V(x)|\leq C/|x|^{\gamma}$, $|x|\leq1$, we get
\begin{equation*}
	\int_{|x|\leq1}|Vu|^{p_{1}}dx=\sum_{j<0}\int_{C(j)}|Vu|^{p_{1}}dx\leq C\sum_{j<0}\int_{C(j)}\left|\frac{u}{|x|^{\gamma}}\right|^{p_{1}}dx.
\end{equation*}
Now, we apply H\"older inequality for
\begin{equation}\label{eq:Holr}
	\frac{1}{p_{1}}=\frac{1}{q}+\frac{2}{n+1}=\frac{2}{n+1}-\left(\frac{1}{2}-\frac{1}{q}\right)+\frac{1}{2}=\frac{1}{r}+\frac{1}{2},
\end{equation} 
and we obtain the following.
\begin{align*}
	\sum_{j<0}\int_{C(j)}\left|\frac{u}{|x|^{\gamma}}\right|^{p_{1}}dx&\leq C\sum_{j<0}\left(\frac{1}{2^{j}}\int_{C(j)}|u|^{2}dx\right)^{\frac{p_{1}}{2}}\left(\int_{C(j)}\frac{dx}{(2^{j})^{(\gamma-\frac{1}{2})r}}\right)^{\frac{p_{1}}{r}}\\
	&
	\leq C\left(\sup_{R\leq1}\frac{1}{R}\int_{B_{R}}|u|^{2}dx\right)^{\frac{p_{1}}{2}}\sum_{j<0}\left(\int_{C(j)}\frac{dx}{(2^{j})^{(\gamma-\frac{1}{2})r}}\right)^{\frac{p_{1}}{r}}.
	\nonumber	
\end{align*}
Then, we have that
\begin{align*}
\|Vu\|_{L^{p_{1}}(|x|\leq1)}&\leq C\left(\sup_{R\leq1}\frac{1}{R}\int_{B_{R}}|u|^{2}dx\right)^{\frac{1}{2}}\sum_{j<0}\left(\int_{C(j)}\frac{dx}{(2^{j})^{(\gamma-\frac{1}{2})r}}\right)^{\frac{1}{r}}\\
&
\leq C\left(\sup_{R\leq1}\frac{1}{R}\int_{B_{R}}|u|^{2}dx\right)^{\frac{1}{2}}\sum_{j<0}(2^{j})^{\frac{n}{r}-(\gamma-\frac{1}{2})}.
\nonumber
\end{align*}
Hence, by imposing 
\begin{equation}\label{eq:con1}
	\frac{n}{r}-(\gamma-\frac{1}{2})>0,
\end{equation}
we get that
\begin{equation*}
	\|Vu\|_{L^{p_{1}}(|x|\leq1)}\leq C\left(\sup_{R\leq1}\frac{1}{R}\int_{B_{R}}|u|^{2}dx\right)^{\frac{1}{2}}.
\end{equation*}
Remind that, from \eqref{eq:Holr}
\begin{equation*}
	\frac{1}{r}=\frac{2}{n+1}-\left(\frac{1}{2}-\frac{1}{q}\right).
\end{equation*}
Therefore, \eqref{eq:con1} is satisfied whenever
\begin{equation}\label{eq:gamma}
	\gamma<\frac{1}{2}+n\left\{\frac{2}{n+1}-(\frac{1}{2}-\frac{1}{q})\right\}.
\end{equation}
Now, we proceed with the term $\|Vu\|_{L^{p_{1}}(|x|\geq1)}$. We have that
\begin{equation*}
	\int_{|x|\geq1}|Vu|^{p_{1}}dx=\sum_{j\geq0}\int_{C(j)}|Vu|^{p_{1}}dx.
\end{equation*}
Now, since $|V(x)|\leq C/|x|^{\delta}$, $|x|\geq1$, we get
\begin{equation*}
	\int_{|x|\geq1}|Vu|^{p_{1}}dx=\sum_{j\geq0}\int_{C(j)}|Vu|^{p_{1}}dx\leq C\sum_{j\geq0}\int_{C(j)}\left|\frac{u}{|x|^{\delta}}\right|^{p_{1}}dx.
\end{equation*}
As before, if we apply H\"older inequality with the exponents
\begin{equation*}
	\frac{1}{p_{1}}=\frac{1}{q}+\frac{2}{n+1}=\frac{2}{n+1}-\left(\frac{1}{2}-\frac{1}{q}\right)+\frac{1}{2}=\frac{1}{r}+\frac{1}{2},
\end{equation*} 
we obtain
\begin{align*}
	\sum_{j\geq0}\int_{C(j)}\left|\frac{u}{|x|^{\delta}}\right|^{p_{1}}dx&\leq C\sum_{j\geq0}\left(\frac{1}{2^{j}}\int_{C(j)}|u|^{2}dx\right)^{\frac{p_{1}}{2}}\left(\int_{C(j)}\frac{dx}{(2^{j})^{(\delta-\frac{1}{2})r}}\right)^{\frac{p_{1}}{r}}\\
	&
	\leq C\left(\sup_{R\geq1}\frac{1}{R}\int_{1<|x|\leq R}|u|^{2}dx\right)^{\frac{p_{1}}{2}}\sum_{j\geq0}\left(\int_{C(j)}\frac{dx}{(2^{j})^{(\delta-\frac{1}{2})r}}\right)^{\frac{p_{1}}{r}}.
	\nonumber	
\end{align*}
Hence, it holds
\begin{align*}
\|Vu\|_{L^{p_{1}}(|x|\geq1)}&\leq C\left(\sup_{R\geq1}\frac{1}{R}\int_{1<|x|\leq R}|u|^{2}dx\right)^{\frac{1}{2}}\sum_{j\geq0}\left(\int_{C(j)}\frac{dx}{(2^{j})^{(\delta-\frac{1}{2})r}}\right)^{\frac{1}{r}}\\
&
\leq C\left(\sup_{R\geq1}\frac{1}{R}\int_{1<|x|\leq R}|u|^{2}dx\right)^{\frac{1}{2}}\sum_{j\geq0}(2^{j})^{\frac{n}{r}-(\delta-\frac{1}{2})}.
\nonumber
\end{align*}
Then, if we assume that 
\begin{equation}\label{eq:con2}
	\frac{n}{r}-(\delta-\frac{1}{2})<0,
\end{equation}
we have that
\begin{equation*}
	\|Vu\|_{L^{p_{1}}(|x|\geq1)}\leq C\left(\sup_{R\geq1}\frac{1}{R}\int_{1<|x|\leq R}|u|^{2}dx\right)^{\frac{1}{2}}.
\end{equation*}
This is satisfied if 
\begin{equation}\label{eq:delta}
	\delta>\frac{1}{2}+n\left\{\frac{2}{n+1}-(\frac{1}{2}-\frac{1}{q})\right\}.
\end{equation}

Now, from \eqref{eq:Vu} and under the assumptions \eqref{eq:gamma} and \eqref{eq:delta} for $V(x)$, we conclude
\begin{align*}
&\|u\|_{L^q(\mathbb{R}^n)}=\|(\Delta+(1\pm i\epsilon))^{-1}(f-Vu)\|_{L^q(\mathbb{R}^n)}\\
	&
	\leq C(\|f\|_{{L^p(\mathbb{R}^n)}}+\|Vu\|_{{L^{p_{1}}(\mathbb{R}^n)}})
	\nonumber
	\\
	&
	\leq C\left(\|f\|_{{L^p(\mathbb{R}^n)}}+\left(\sup_{R\leq1}\frac{1}{R}\int_{B_{R}}|u|^{2}dx\right)^{\frac{1}{2}}+\left(\sup_{R\geq1}\frac{1}{R}\int_{1<|x|\leq R}|u|^{2}dx\right)^{\frac{1}{2}}\right).
	\nonumber
\end{align*}
Finally, reminding that $\triple{u}$ is still bounded for  $\left(\frac{1}{p},\frac{1}{q}\right)\in\Delta(n)$ such that $\frac{1}{2n}<\frac{1}{2}-\frac{1}{q}<\frac{1}{n+1}$, we can conclude the final estimate
\begin{equation*}
		\|u\|_{L^q(\mathbb{R}^n)}\leq C \|f\|_{L^p(\mathbb{R}^n)}.
\end{equation*}
Now, by applying duality we have that the result is true whenever $\left(\frac{1}{p},\frac{1}{q}\right)\in\Delta(n)$ such that $\frac{1}{2n}<\frac{1}{p}-\frac{1}{2}<\frac{1}{n+1}$.
This ends the proof.
\end{proof}
\begin{remark}
Notice that the necessary decay for $V$ given by $\delta$ in \eqref{eq:delta} grows as we approach the upper frontier of $\Delta(n)$ (See line segment $AB$ in Figure \ref{fig:deltan0med2}). However, we can always take for $|x|\geq1$, $|V(x)|\leq C/|x|^\alpha$, $\alpha<2$, where $C$ is  not necessarily small.
\end{remark}

\begin{acknowledgments}
This work is part of the author's Ph.D. thesis. The author is supported by the Academy of Finland (Centre of Excellence in
Inverse Problems Research, grant number 250 215) and belongs to the project MTM2011-24054 Ministerio de Ciencia y Tecnolog\'ia de Espa\~na.
\end{acknowledgments}

\end{document}